\documentclass[reqno,11pt]{amsart}
\usepackage{latexsym,amssymb,amsfonts,amsmath,amsthm}
\usepackage{mathrsfs}
\usepackage[usenames]{color}
\usepackage{eucal}
\usepackage{graphicx}

\newtheorem{thm}{Theorem}[section]
\newtheorem{cor}[thm]{Corollary}
\newtheorem{lem}[thm]{Lemma}
\newtheorem{prop}[thm]{Proposition}
\theoremstyle{definition}

\theoremstyle{remark}
\newtheorem{rem}[thm]{Remark}
\numberwithin{equation}{section}
\newcommand{\Real}{\mathbb R}
\newcommand{\eps}{\varepsilon}

\newcommand{\F}{\mathcal{F}}

\newcommand{\one}[1]{\mathbf{1}_{\{#1\}}}

\newcommand{\One}[1]{\text{\Large \bf 1}\left\{#1\right\}}
\renewcommand{\P}{\mathsf{P}}
\newcommand{\Q}{\mathsf{Q}}
\newcommand{\PP}{\mathbb{P}}
\newcommand{\EE}{\mathbb{E}}
\newcommand{\E}{\mathsf{E}}

\begin{document}

\title[]{The Euler-Maruyama approximation for the absorption time of the CEV diffusion}%
\author{P. Chigansky}%
\address{Department of Statistics,
The Hebrew University,
Mount Scopus, Jerusalem 91905,
Israel}
\email{pchiga@mscc.huji.ac.il}

\author{F.C. Klebaner}
\address{ School of Mathematical Sciences,
Monash University Vic 3800,
Australia}
\email{fima.klebaner@monash.edu}

\thanks{ Research supported by the  Australian Research Council Grants DP0881011 and
DP0988483}


\date{9, May, 2011}%
\begin{abstract}
A standard convergence analysis of the simulation schemes for the
hitting times of diffusions typically requires  non-degeneracy of
their coefficients on the boundary, which excludes the possibility
of absorption. In this paper we consider the CEV diffusion from the
mathematical finance and show how a weakly consistent approximation
for the absorption time can be constructed, using the Euler-Maruyama scheme.
\end{abstract}
\maketitle

\section{Introduction}

In scientific computations, expectations with respect to probabilities, induced by
continuous time processes, are often replaced by Monte Carlo averages over independent trajectories.
For diffusions, generated by stochastic differential equations (SDEs), the trajectories are usually  approximated numerically
(see e.g. \cite{KKK10}).
The accuracy assessment of such numerical procedures is  a well studied topic and
the available theory establishes and quantifies the convergence of the approximations to the actual solution in a variety of
modes, depending on the properties of the SDE coefficients. This in turn typically suffices to claim convergence of expectations
for path functionals,  continuous in an appropriate topology, but unfortunately, may not apply to
discontinuous functionals, some of which arise naturally in applications.

One important example of such a functional is the hitting time of a domain boundary.
Let $X=(X_t)_{t\ge 0}$ be
the diffusion process on $\Real$, generated by the It\^o SDE
\begin{equation}
\label{gensde}
dX_t = b(X_t)dt + a(X_t)dB_t, \quad X_0=x \in \Real,
\end{equation}
where $B=(B_t)_{t\ge 0}$ is the Brownian motion and the coefficients $b(\cdot)$ and $a(\cdot)$ are functions,
assumed to satisfy the regularity conditions, guaranteeing existence of the unique strong solution
(see e.g. \cite{RW1,RW2}). The hitting time of the level $\ell\in \Real$ is
$$
\tau_\ell(X) := \inf\{t\ge 0: X_t=\ell\},
$$
where  $\inf\{\emptyset\}=\infty$. Thus $\tau_\ell(X)$  is an extended random variable with
values in the Polish space $\bar\Real_+:=\Real_+\cup \{\infty\}$,  endowed with the
metric $\rho(s,t)=|\arctan(s)-\arctan(t)|$, $s,t\in \bar\Real_+$.

Consider a family of continuous processes $X^\delta=(X^\delta_t)_{t\ge 0}$, generated by
a numerical scheme with the time step parameter $\delta>0$
(such as e.g. the Euler-Maruyama recursion \eqref{em} and \eqref{lin} below)
and suppose that $X^\delta$ approximates the diffusion $X$ in the sense of weak convergence.
More precisely, let $C([0,\infty),\Real)$ be the space of real valued continuous functions on $[0,\infty)$,
endowed with the metric
\begin{equation}
\label{metric}
\varrho(u,v)=\sum_{j\ge 1} 2^{-j} \sup_{t\le j}|u_t-v_t|, \quad u,v\in C([0,\infty),\Real).
\end{equation}
Then $X^\delta$ converges weakly to $X$, if for any bounded and continuous functional $h: C([0,\infty),\Real)\mapsto \Real$
\begin{equation}
\label{wconv}
\lim_{\delta\to 0}\E h(X^\delta)=\E h(X).
\end{equation}
Such convergence can often be established using the techniques, developed in  e.g. \cite{KE86}, \cite{JS03}, \cite{LS89}.

In particular, if $\phi: C([0,\infty),\Real)\mapsto \bar \Real_+$ is  a functional, almost surely continuous with respect to the
measure induced by $X$,  then \eqref{wconv}
implies
\begin{equation}
\label{phiwc}
\lim_{\delta\to 0} \E f\big(\phi(X^\delta)\big) =  \E f\big(\phi(X)\big)
\end{equation}
for any  continuous and bounded function $f:\bar \Real_+\mapsto \Real$. In other words, the weak convergence
of the processes $X^\delta\to X$ implies the weak convergence of the random variables $\phi(X^\delta)\to \phi(X)$ or, equivalently,
the convergence of the probability distribution functions  $\lim_{\delta\to 0}\P(\phi(X^\delta)\le x)=\P(\phi(X)\le x)$
for any  $x\in \bar \Real_+$, at which the distribution function $\P(\phi(X)\le x)$ is continuous.

Let us now take a closer look at the hitting time
$$
\tau_\ell (u) =\inf\{t\ge 0: u_t=\ell\}, \quad u\in C([0,\infty),\Real),
$$
viewed as a functional on $C([0,\infty),\Real)$. Clearly $\tau_\ell(\cdot)$ is discontinuous
at some paths: for example, if $u_t = \max(1-t,0)$ and $u^j_t = \max(1-t,1/j)$, then
$\varrho(u_t,u_t^j)\le1/j\to 0$ as $j\to\infty$, but $\tau(u)=1$ and $\tau(u^j)=\infty$ for all $j$.

On the other hand, $\tau_\ell(\cdot)$ is continuous at any $u$, which either upcrosses $\ell$:
$$
\tau_{\ell+}(u):= \inf\{t\ge 0: u_t >\ell\} = \tau_\ell(u), \quad \text{if\ } u_0\le  \ell
$$
or downcrosses $\ell$:
$$
\tau_{\ell-}(u):= \inf\{t\ge 0: u_t <\ell\} = \tau_\ell(u), \quad \text{if\ } u_0\ge \ell.
$$

If this type of paths is typical for the diffusion \eqref{gensde}, i.e.
if the set of all such paths is  of full measure, induced by the process $X$, then $\tau_\ell(\cdot)$ is essentially continuous
and the weak convergence $X^\delta\to X$ still implies the weak convergence $\tau_\ell(X^\delta)\to \tau_\ell(X)$.
However, if $\ell$ is an accessible absorbing boundary of $X$, the paths which hit  $\ell$ cannot leave it at any further time.
For such diffusions $\tau_\ell(\cdot)$ is discontinuous on a set of positive  probability and the weak convergence
$\tau_\ell(X^\delta)\to \tau_\ell(X)$ cannot be directly deduced from  $X^\delta\to X$.

Hitting times play an important role in applied sciences, such as physics or finance, and since their exact probability distribution
cannot be found in a closed form beyond special cases, practical approximations are of considerable
interest. There are two principle approaches to compute such approximations. One is based on the fact
that the expectation of a given function of the hitting time solves the Dirichlet boundary problem for an appropriate PDE, and
thus the approximations can be computed using the generic tools from the PDE numerics.
Sometimes, the particular structure of the emerging PDE can be exploited to calculate expectations of special functions of hitting times, such as  moments (as e.g. in  the linear programming approach of \cite{HRS01}).

The probabilistic
alternative is to use the Monte Carlo simulations, in which the diffusion paths are approximated by numerical solvers.
Typically the diffusions are simulated on a discrete grid of points and the evaluation of the hitting times requires
construction of the continuous paths through an interpolation. The naive approach  is to use the general purpose interpolation techniques,
such as the one used in our paper (see \eqref{lin} below). Better results are obtained if the possibility of having a hit
between the grid points is taken into account as in e.g. \cite{LL94}, \cite{GS99}, \cite{GSZ01}, \cite{JL03}.

The convergence analysis of the approximations of the hitting times, based on the various numerical schemes,
appeared in \cite{M95}, \cite{MT99,MT02}, \cite{GM04}, \cite{G00}. The results obtained in these
works, assume ellipticity or hypoellipticity of the diffusion processes under consideration, which corresponds to
the case of non-absorbing boundary in the preceding discussion. The analysis beyond these non-degeneracy conditions appears to be
a much harder problem.

In this paper we consider a particular diffusion on $\Real_+$, with an absorbing boundary at $\{0\}$.
As explained above, the absorption time in this case is a genuinely discontinuous functional of the
diffusion paths, which makes the convergence analysis of the approximations a delicate
matter. We propose a simple approximation procedure, based on the Euler-Maruyama scheme, and prove its weak consistency.

In the next section we formulate the precise setting of the problem and state the main result, whose prove is
given in Section \ref{sec-3}. The results of  numerical simulations are gathered in Section \ref{sec-num}
and some supplementary calculations are moved to the appendices.

\section{The setting and the main result}


Consider the diffusion process $X=(X_t)_{t\ge 0}$, generated by the It\^o SDE
\begin{equation}\label{cev}
dX_t = \mu X_t dt + \sigma X^p_t dB_t, \quad x\in \Real_+, \quad t\ge 0,
\end{equation}
where $\mu\in \Real$, $\sigma>0$ and $p\in [1/2,1)$ are constants and $B=(B_t)_{t\ge 0}$
is the Brownian motion, defined on a filtered probability space $(\Omega,\F,(\F_t), \P)$, satisfying the
usual conditions. This SDE has the unique strong solution (Proposition 1 in \cite{AKL10}) and
is known in mathematical finance as the Constant Elasticity of Variance (CEV) model (see e.g. \cite{DK01}).
For $p=1/2$,  it  is also Feller's branching diffusion, being the weak limit of the Galton\,-\,Watson
branching processes under appropriate scaling.

The process $X$  is a regular diffusion on $\Real_+\cup\{0\}$ and a standard calculation reveals
that $\{0\}$  is an absorbing (or Feller's exit) boundary (see \S 6, Ch. 15 in \cite{KT2}).
We will denote by  $(\P_x)_{x\in \Real_+}$ the corresponding family of Markov probabilities with $\P_x(X_0=x)=1$, induced by $X$
on the measurable space $C([0,\infty),\Real_+)$ with the metric \eqref{metric}.

Consider the continuous time process $X^\delta=(X^\delta_t)$, $t\ge 0$, which satisfies
the Euler-Maruyama recursion at the grid points $t_j\in \delta \mathbb{Z}_+$
\begin{equation}
\label{em}
X^\delta_{t_{j}}= X^\delta_{t_{j-1}} + \mu\big(X^\delta_{t_{j-1}}\big)^+ \delta + \sigma \big(X^{\delta}_{t_{j-1}}\big)^{+p} \xi_j\sqrt{\delta}, \quad X_0=x>0
\end{equation}
and is piecewise linear otherwise:
\begin{equation}\label{lin}
X^\delta_t = \frac{t-t_{j-1}}{\delta}X^\delta_{t_{j}}  + \frac{t_{j}-t}{\delta} X^\delta_{t_{j-1}}, \quad t\in [t_{j-1},t_{j}],
\end{equation}
where $x^+:=\max(0,x)$,  $\delta>0$ is a small time step parameter and $(\xi_j)_{j\in \mathbb{Z}_+}$ is a sequence of i.i.d.
$N(0,1)$ random variables.

Since the diffusion coefficient of \eqref{cev} degenerates and is not Lipschitz on the boundary $\{0\}$, this SDE does not quite
fit the standard numerical frameworks such as  \cite{KP} or \cite{MT04}. Nevertheless the scheme \eqref{em} does approximate the solution of \eqref{cev}
in the sense of the weak convergence of measures, as was recently shown in
\cite{AKL10} (see also \cite{Z08}, \cite{Y02}, \cite{R09}). Consequently, for any  $\P_x$-a.s. continuous
functional $\phi: C\big([0,\infty),\Real\big)\mapsto \Real_+\cup\{\infty\}$
\begin{equation}\label{weak}
\phi(X^\delta)\xrightarrow[\delta\to 0]{w} \phi(X),
\end{equation}
where $\stackrel{w}{\to}$ stands for the the weak convergence,  defined in \eqref{phiwc}.
Since a typical trajectory of $X$ oscillates around the level $a>0$ after crossing it, the functional  $\tau_a(\cdot)$ is $\P_x$-a.s. continuous for $a>0$ (see Lemma \ref{lem1} below) and hence $\tau_a(X^\delta)\stackrel{w}{\to} \tau_a(X)$  as $\delta\to 0$.

This argument, however,
does not apply to $\tau_0(\cdot)$, since it is {\em essentially} discontinuous, as discussed in the Introduction.
Leaving the question of convergence $\tau_0(X^\delta)\stackrel{w}{\to} \tau_0(X)$ open, we shall prove the following result

\begin{thm}\label{thm}
For any  $\beta \in \left(0,\frac {1/2} {1-p}\right)$,
\begin{equation}
\label{main}
\tau_{\delta^\beta}(X^\delta)\xrightarrow{w} \tau_0(X), \quad \delta\to 0.
\end{equation}

\end{thm}

Note that $\tau_{\delta^\beta}(\cdot)$ is a continuous functional for any fixed $\delta>0$
and hence can be seen as a mollified version of the discontinuous $\tau_0(\cdot)$. The parameter $\beta$
controls the mollification, relatively to the step-size parameter of the Euler-Maruyama algorithm.

In practical terms,  the convergence \eqref{main} provides theoretical justification for  the procedure, in which the approximate trajectory $X^\delta$,
generated by \eqref{em} and \eqref{lin}, is stopped not at $0$, but at $\delta^\beta>0$, which only approaches  zero as $\delta\to 0$.
Our method does not quantify the convergence in terms of e.g. rates, but the numerical experiments in Section
\ref{sec-num} indicate that this stopping rule produces practically adequate results.

\begin{rem}
As will become clear from the proof, our approach exploits the local behavior of the SDE coefficients
near the  boundary and can be applied to the more general one-dimensional diffusions of the form \eqref{gensde},
for which a weakly convergent numerical scheme is available. For example, all the arguments in the proof of
Theorem \ref{thm} directly apply to the diffusion, whose
coefficients $b(x)$ and $a(x)$ have a similar local asymptotic at $0$ as the CEV model.
The SDE \eqref{cev} is a  study case, which seems to capture the essential difficulties of the problem, related to
the degeneracy of the SDE coefficients on the absorbing boundary.
It is a convenient choice,  since the weak convergence \eqref{weak}, being the starting point of our approach,
has been already established for the CEV model in \cite{AKL10}, and the explicit formulas  for the probability
density of $\tau_0(X)$ are available,  allowing to carry out the numerical experiments.
\end{rem}

\section{The proof of Theorem \ref{thm}} \label{sec-3}

The proof, inspired by the approach of S.Ethier \cite{E79}, is  based on the following
observation  (a variation of Billingsley's lemma, see Proposition 6.1 \cite{E79})

\begin{prop}\label{lem:Bil}
Let $\Q^n$, $n\in \mathbb{N}$ and $\Q$ be Borel probability measures on a metric space $S$ such that
$\Q^n\xrightarrow{w}\Q$ as $n\to\infty$. Let $S'$ be a separable metric space with metric $\rho$,
and suppose that $\phi_k$, $k\in \mathbb{N}$ and $\phi$ are Borel measurable
mappings of $S$ into $S'$ such that
\begin{enumerate}
\renewcommand{\theenumi}{\roman{enumi}}
\item\label{i} $\phi_k$ is $\Q$-a.s. continuous for all $k\in \mathbb{N}$
\item\label{ii} $\lim_k \phi_k= \phi$, $\Q$-a.s.
\item\label{iii} for every $\eta>0$,
$$
\lim_{k}\varlimsup_{n}\Q^n\Big(\rho(\phi_k,\phi_{\alpha_n})>\eta\Big) =0.
$$
for an increasing real sequence $(\alpha_n)$.
\end{enumerate}
Then $\Q^n\circ \phi^{-1}_{\alpha_n} \xrightarrow{w} \Q\circ \phi^{-1}$ as $n\to\infty$.
\end{prop}

\begin{proof}
Let $h$ be a continuous with respect to $\rho$ bounded real valued function on $S'$, then
\begin{multline*}
\big|\E^n h(\phi_{\alpha_n})-\E h(\phi)\big|\le
\E^n\big| h(\phi_{\alpha_n})- h(\phi_k)\big|
+ \\
\big|\E^n h(\phi_k)-\E h(\phi_k)\big|
+
\E\big| h(\phi_k)- h(\phi)\big|,
\end{multline*}
where $\E$ and $\E^n$ denote the expectations with respect to $\Q$ and $\Q^n$ respectively.
Since $h$ is continuous and $\phi_k\to \phi$, $\Q$-a.s., the last term vanishes as $k\to\infty$ by the dominated convergence.
Moreover, since for any fixed $k$,
$\phi_k$ is $\Q$-a.s. continuous and $\Q^n\xrightarrow{w}\Q$ as $n\to\infty$,
$$
\lim_n \big|\E^n h(\phi_k)-\E h(\phi_k)\big|=0, \quad \forall k.
$$
Consequently,
$$
\varlimsup_k\varlimsup_n \big|\E^n h(\phi_{\alpha_n})-\E h(\phi)\big|\le
\varlimsup_k\varlimsup_n\E^n\big| h(\phi_{\alpha_n})- h(\phi_k)\big|=0
$$
where the latter equality holds by \eqref{iii}.
The claim follows by arbitrariness of $h$.
\end{proof}

Let us now outline the plan  of the proof.
In our context, the probability measures, induced by the family  $X^\delta$, play the role of $\Q^\delta$
and by Proposition \ref{prop} they converge weakly to the law $\Q:=\P_x$ of the diffusion $X$.
Since the diffusion coefficient of \eqref{cev} is positive, away from the boundary point 0,
$\tau_\eps(\cdot)$  is  a $\P_x$-a.s. continuous functional (Lemma \ref{lem1}) and hence \eqref{i} of Proposition \ref{lem:Bil}
holds.
On the other hand, $\lim_{\eps\to 0}\tau_\eps(u)= \tau_0(u)$ for any $u\in C([0,\infty),\Real)$ (Lemma \ref{lem2}),
which implies \eqref{ii} of Proposition \ref{lem:Bil}. The more intricate part is the convergence \eqref{iii}, which is
verified in Lemma \ref{lem3}, using the particular structure of the SDE \eqref{cev}.
The statement of Theorem \ref{thm} then follows from Proposition \ref{lem:Bil}.

\vskip 0.1in

The following result is essentially proved in \cite{AKL10}:
\begin{prop}\label{prop}
The processes $(X^\delta)$, defined by \eqref{em} and \eqref{lin},
converge weakly to the diffusion $X$, defined by \eqref{cev}, as $\delta\to 0$.
\end{prop}

\begin{proof}
For the process, obtained by piecewise constant
interpolation of the points generated by the recursion \eqref{em}, the claim is established in
Theorem 1.1 in \cite{AKL10}. The extension to the piecewise linear interpolation
\eqref{lin} is straightforward.
\end{proof}

\begin{lem}[Proposition 4.2. in \cite{E79}] \label{lem1}
For all $x\ge \eps>0$, $\tau_\eps(\cdot)$ is a $\P_x$-a.s. continuous functional on $C([0,\infty),\Real)$.
\end{lem}

\begin{proof}
We shall prove the claim for completeness and the reader's convenience.
For $u\in C([0,\infty),\Real)$ and $\eps>0$, let
$$
\tau_{\eps-}(u):=\inf\{t\ge 0: u_t<\eps\}
$$
and define
$$
A:=\left\{u\in C([0,\infty), \Real):  \tau_\eps(u) = \tau_{\eps-}(u)\right\}.
$$
We shall first show that $u\mapsto \tau_\eps(u)$ is continuous on $A$, i.e. that
\begin{equation}
\label{eq:3.1}
\varrho(u^n, u)\to 0\quad \implies\quad \rho\big(\tau_\eps(u^n), \tau_\eps(u)\big)\to 0,\quad \forall u\in A,
\end{equation}
and then check that
\begin{equation}
\label{eq:3.1'}
\P_x(X\in A)=1\quad \forall \; x\ge \eps.
\end{equation}

To this end, note that if $u\in A$ and  $\tau_\eps(u)=\infty$, then $\min_{t\le T}u_t>\eps$ for all $T>0$
(recall $u_0\ge \eps$ for $u\in A$).  Thus $\varrho(u^n,u)\to 0$
implies $\lim_n \min_{t\le T} u^n_t >\eps$ and hence $\varliminf_{n}\tau_\eps(u^n)\ge T$. Since $T$ is arbitrary, $\lim_n \tau_\eps(u^n)=\infty$,
i.e.  \eqref{eq:3.1} holds.

Now take $u\in A$, such that  $\tau_\eps(u)<\infty$. If  $\tau_\eps(u)=0$, the claim obviously holds by continuity of $u$. If $\tau_\eps(u)>0$,  fix an $\eta>0$ such that $\tau_\eps(u)-\eta>0$.
Since
$$
\min_{ t\le \tau_\eps(u)-\eta} u_t >\eps
\quad\text{and}\quad
\sup_{t\le \tau_\eps(u)-\eta}|u^n_t-u_t|\xrightarrow{n\to\infty} 0
$$
we have
$
\lim_n \min_{ t\le \tau_\eps(u)-\eta} u^n_t >\eps
$
and thus $\varliminf_n \tau_\eps(u^n)\ge \tau_\eps(u)-\eta$.
On the other hand, as $\tau_\eps(u)=\tau_{\eps-}(u)$, for any $\eta>0$, there is an $r\le \tau_\eps(u)+\eta$, such that $u_r<\eps$.
It follows that $\lim_n u_n(r)<\eps$ and hence $\varlimsup_n \tau_\eps(u_n)\le \tau_\eps(u)+\eta$.
By arbitrariness of $\eta$, we conclude that $\lim_n \tau_\eps(u_n)=\tau_\eps(u)$ and \eqref{eq:3.1} holds.

It is left to show that \eqref{eq:3.1'} holds. The diffusion $X$ satisfies the strong Markov property and thus (we write
$\tau_\eps$ for $\tau_\eps(X)$),
$$
\P_x(A)=\P_x\big(\tau_\eps= \tau_{\eps-}\big)=
\E_x \one{\tau_\eps<\infty}\P_{\eps}( \tau_{\eps-}=0)
+\P_x(\tau_\eps=\infty),
$$
and hence the required claim holds, if $\P_\eps( \tau_{\eps-}=0)=1$.

Take now $f$ to be a scale function of the diffusion $X$, i.e. a solution to the equation
$ \mathcal{L}f=0$,  where $\mathcal{L}$ is the generator of $X$:
\begin{equation}
\label{genX}
\big(\mathcal{L}\psi\big)(x) = \mu x \psi'(x)+ \frac 1 2\sigma^2 x^{2p} \psi''(x), \quad x >0
\end{equation}
It is well known, e.g. \cite{Kle},
that we can take it to be positive and increasing, specifically, for \eqref{genX},
$$
f(x):=\int_0^x \exp\left(-\int_{0}^y\frac{2\mu z}{(\sigma z^{p})^2}dz\right)dy
$$
can be taken.
The process $f(X_t)$ is a nonnegative local martingale and thus a supermartingale (e.g. \cite{Kle} p. 197).
The random variable
$t\wedge \tau_{\eps-}$ is a bounded stopping time and by the optional stopping theorem  we have for any $t\ge 0$
\begin{equation}\label{stopSf}
\E_\eps f(X_{t\wedge \tau_{\eps-}}) \le f(\eps).
\end{equation}
By the definition of $ \tau_{\eps-}$ and path continuity of $X$, it follows that
$X_{t\wedge \tau_{\eps-}}\ge \eps$, and
since $f$ is increasing,  we have that $f(X_{t\wedge \tau_{\eps-}})\ge f(\eps)$. Thus it follows from \eqref{stopSf} that  $P_\eps(X_{t\wedge \tau_{\eps-}}=\eps)=1$ for all $t\ge 0$ and, consequently,
$[X,X]_{t\wedge \tau_{\eps-}}=0$, $\P_\eps$-a.s. for $t\ge 0$.
On the other hand, $[X,X]_{\tau_{\eps-}}=\int_0^{\tau_{\eps-}}\sigma^2 X_s^{2p}ds>0$, on the set  $\big\{\tau_{\eps-}>0\big\}$, $\P_\eps$-a.s.
The obtained contradiction  implies $\P_\eps( \tau_{\eps-}>0)=0$, as claimed.

\end{proof}

\begin{lem}\label{lem2}
$
\lim_{\eps\to 0}\tau_\eps(u)=\tau_0(u)
$
for any $u\in C([0,\infty),\Real_+)$.
\end{lem}

\begin{proof}
If $\tau_0(u)=\infty$, then $\min_{t\le T} u_t>0$ and hence $\varliminf_{\eps\to 0}\tau_\eps(u)\ge T$ and the claim follows,
since $T$ is arbitrary. If $\tau_0(u)<\infty$, then for an $\eta>0$, $\inf_{t\le \tau_0(u)-\eta}u_t>0$ and thus
$\varliminf_{\eps\to 0}\tau_\eps(u)\ge \tau_0(u)-\eta$. For sufficiently small $\eps>0$,  $\tau_\eps(u)\le \tau_0(u)$ and the claim follows.
\end{proof}

Let $\PP$ be the probability on the space, carrying the sequence $(\xi_j)$ (see \eqref{em})
and denote by $(\PP_x)_{x\ge 0}$ the Markov family of probabilities, corresponding to  the discrete time process
$(X^\delta_{t_j})$ with  $\PP_x(X^\delta_0=x)=1$. Since the process $(X^\delta_t)$ is piecewise linear off the grid $\delta \mathbb{Z}_+$,
the condition \eqref{iii} of Proposition \ref{lem:Bil} follows from

\begin{lem}\label{lem3}
For any  $\beta \in \left(0,\frac{1/2}{1-p}\right)$  and  $\eta>0$
\begin{equation}\label{eq:2.5}
 \lim_{\eps\to 0}\varlimsup_{\delta\to 0}
\PP_x\Big(\rho\big(\tau_\eps(X^\delta), \tau_{\delta^\beta}(X^\delta)\big)> \eta\Big)=0.
\end{equation}
\end{lem}

\begin{proof}

Roughly speaking, \eqref{eq:2.5} means that a trajectory of $X^\delta$, which approaches the boundary $\{0\}$,
is very likely to hit it. This seemingly plausible statement is not at all obvious, since the coefficients of
our diffusion decrease to zero near this boundary, making it hard to reach. By letting the level $\delta^\beta$
decrease to zero at a particular rate allows to approximate expectations of the
hitting times $\tau_{\delta^\beta}(X^\delta)$ by those of $\tau_{\delta^\beta}(X)$, which in turn can be estimated
using their relations to the corresponding boundary value problems.

In what follows, $C$, $C_1$,$C_2$ etc. denote  unspecified constants, independent of $\eps$ and $\delta$, which   may  be different in each appearance.
Define the crossing times of level $a$
$$
\nu_{a} = \delta\; \inf\left\{j\ge 1: a\in \big(X^\delta_{t_{j-1}},X^\delta_{t_j}\big] \right\}, \quad a\in \Real_+,
$$
with $\inf\{\emptyset\}=\infty$.
The sequence  $(X^\delta_{t_j})$, $j\in \mathbb{Z}_+$ is a strong Markov process and
$\nu_a/\delta$ is a stopping time with respect to its natural filtration.

Since $X^\delta$ is piecewise linear,
$$
|\tau_\eps(X^\delta)-\nu_\eps|\le \delta \quad \text{and}\quad |\tau_{\delta^\beta}(X^\delta)-\nu_{\delta^\beta}|\le \delta
\quad \text{on the set\ } \left\{\tau_\eps(X^\delta)<\infty\right\}
$$
and thus by the triangle inequality
$$
\rho\big(\tau_\eps(X^\delta), \tau_{\delta^\beta}(X^\delta)\big)\le
\rho\big(\nu_\eps, \nu_{\delta^\beta}\big) +2\delta.
$$
Since
$
\tau_\eps(X^\delta)\le \tau_{\delta^\beta}(X^\delta)
$
for  $x\ge \eps\ge \delta^\beta$,
it follows
\begin{align*}
&\PP_x\Big(\rho\big(\tau_\eps(X^\delta), \tau_{\delta^\beta}(X^\delta)\big)> \eta\Big)=
\PP_x\Big(\rho\big(\tau_\eps(X^\delta), \tau_{\delta^\beta}(X^\delta)\big)> \eta, \tau_\eps(X^\delta)<\infty\Big)\le  \\
&\PP_x\Big(\rho\big(\nu_\eps, \nu_{\delta^\beta}\big)> \eta-2\delta, \nu_\eps <\infty\Big)=
\EE_x \one{\nu_\eps <\infty}\PP_{X^\delta_{\nu_\eps}}\Big(\rho\big(0, \nu_{\delta^\beta}\big)> \eta-2\delta\Big) \le\\
&\sup_{z\in [0,\eps]}\PP_{z}\big(\nu_{\delta^\beta}> \eta'\big),
\end{align*}
where $\eta':= \tan(\eta/2)$ (assuming $\delta$ is small enough).
Further,
\begin{align*}
&\PP_{z}\big(\nu_{\delta^\beta}> \eta'\big)=
\PP_{z}\big(\nu_{\delta^\beta}> \eta', \nu_{1}>\eta'\big)
+
\PP_{z}\big(\nu_{\delta^\beta}> \eta', \nu_{1}\le \eta'\big) \le \\
&\PP_{z}\big(\nu_{\delta^\beta}\wedge  \nu_{1}>\eta'\big)
+
\PP_{z}\big(\nu_{\delta^\beta}>\nu_{1}\big)\le \frac 1 {\eta'}
\EE_z\big(\nu_{\delta^\beta}\wedge  \nu_{1}\big)
+
\PP_{z}\big(\nu_{\delta^\beta}>\nu_{1}\big),
\end{align*}
and thus \eqref{eq:2.5} holds, if we show
\begin{equation}
\label{une}
\lim_{\eps\to 0}\varlimsup_{\delta\to 0}\sup_{z\in [0,\eps]}\EE_z\Big(\nu_{\delta^\beta}\wedge  \nu_{1}\Big)=0,
\end{equation}
and
\begin{equation}
\label{doux}
\lim_{\eps\to 0}\varlimsup_{\delta\to 0}\sup_{z\in [0,\eps]}\PP_{z}\Big(\nu_{\delta^\beta}>\nu_{1}\Big)=0.
\end{equation}

\subsubsection*{Proof of \eqref{une}}

We shall use the regularity properties of the function
$$\psi(x):=\E_x \tau_0(X) \wedge \tau_{1}(X), \quad x\in [0,1]$$ near the boundary point 0,
summarized in the Appendix \ref{A}.  In particular, $\psi$ is continuous on the interval
$[0,1]$ and is smooth on $(0,1]$. Note, however, that the derivatives of $\psi$ explode at the boundary point $0$,
which is related to the possibility of absorption.
We shall extend the domain of $\psi$ to the whole $\Real$
by continuity, setting $\psi(x)=0$ for $x\in \Real\setminus [0,1]$.

Consider the Taylor expansion
\begin{align*}
\psi(X^\delta_{t_n})-\psi(X^\delta_0) =& \sum_{j=1}^n \psi'(X^\delta_{t_{j-1}})\big(\mu X^\delta_{t_{j-1}}\delta +
\sigma \big(X^\delta_{t_{j-1}}\big)^p \xi_j \sqrt{\delta}
\big) +\\
&
\sum_{j=1}^n \frac 1 2 \psi''(X^\delta_{t_{j-1}})\big(\mu X^\delta_{t_{j-1}}\delta +
\sigma  \big(X^\delta_{t_{j-1}}\big)^p \xi_j \sqrt{\delta}
\big)^2 + \\
&
\sum_{j=1}^n \frac 1 {3!} \psi'''(\tilde X^\delta_{t_{j-1}})\big(\mu X^\delta_{t_{j-1}}\delta +
\sigma  \big(X^\delta_{t_{j-1}}\big)^p \xi_j \sqrt{\delta}
\big)^3
\end{align*}
where $\tilde X^\delta_{t_{j-1}}$ is between $X^\delta_{t_{j-1}}$ and $X^\delta_{t_j}$. After rearranging terms, the latter reads
$$
\psi\big(X^\delta_{t_n}\big) - \psi(X^\delta_0) =
\sum_{j=1}^n
\big(\mathcal{L} \psi\big)\big(X^\delta_{t_{j-1}}\big)\delta + M_n + R_n
$$
where $\mathcal{L}$  is the generator  defined in \eqref{genX},
the second term is the martingale
$$
M_n := \sum_{j=1}^n \bigg(
\sigma \big(X^\delta_{t_{j-1}}\big)^p \psi'\big(X^\delta_{t_{j-1}}\big)\xi_j \sqrt{\delta}
+\frac 1 2 \sigma^2 \big(X^\delta_{t_{j-1}}\big)^{2p}\psi''\big(X^\delta_{t_{j-1}}\big)\big(\xi_j^2-1\big)\delta
\bigg),
$$
and the last term is the residual
\begin{multline*}
R_n = \sum_{j=1}^n
\frac 1 2 \psi''(X^\delta_{t_{j-1}})\mu^2 \big(X^\delta_{t_{j-1}}\big)^2\delta^2+
\sum_{j=1}^n \psi''(X^\delta_{t_{j-1}})\mu \sigma  \big(X^\delta_{t_{j-1}}\big)^{p+1} \xi_j \delta^{3/2}
+
\\
\sum_{j=1}^{n}   \frac 1 {3!}\psi'''\big(\tilde X^\delta_{t_{j-1}}\big)
\Big(\mu X^\delta_{t_{j-1}}\delta + \sigma \big(X^\delta_{t_{j-1}}\big)^p\xi_j\sqrt{\delta}\Big)^3:=
R^{(1)}_n + R^{(2)}_n+R^{(3)}_n.
\end{multline*}
Consequently, for an integer $k\ge 1$ and the stopping time $\nu:= \nu_{\delta^\beta}\wedge \nu_1$,
\begin{multline}
\label{eq:2.9}
\psi\big(X^\delta_{\nu \wedge k \delta}\big) - \psi(X^\delta_0) =
\sum_{j=1}^{(\nu /\delta\wedge k)-1 }
\big(\mathcal{L} \psi\big)\big(X^\delta_{t_{j-1}}\big)\delta +\\
 M_{\nu/\delta \wedge k} +
R_{(\nu/\delta \wedge k)-1} +r(\delta),
\end{multline}
where
$$
r(\delta):=
\psi\big(X^\delta_{\nu\wedge k\delta }\big) -\psi\big(X^\delta_{(\nu\wedge k\delta) -\delta}\big)
-\Big(M_{\nu/\delta \wedge k}-M_{(\nu/\delta \wedge k)-1}\Big),
$$
is the residual term, which accommodates the possible overshoot at the terminal crossing time $\nu$.

Recall that $\mathcal{L}\psi = -1$ for $x\in (0,1)$ and hence
\begin{equation}
\label{Sbnd}
\sum_{j=1}^{\nu/\delta\wedge k-1}
\big(\mathcal{L} \psi\big)(X^\delta_{t_{j-1}}) \delta = - (\nu\wedge k\delta)+\delta.
\end{equation}
By Lemma \ref{lem:A2}, $\sup_{x\in (0,1)}\Big(x^p|\psi'(x)|+x^{2p} |\psi''(x)|\Big)\le C$ and hence  the increments of $M_n$ satisfy
$$
\EE_z\Big(|M_j-M_{j-1}|\big|\F^{\xi}_{j-1}\Big) \le C<\infty, \quad \text{on\ } \{j\le \nu/\delta\wedge k\}
$$
and  by the optional stopping theorem (Theorem 2 of  \S 2, Ch. VII, \cite{Sh}) $\EE_z M_{\nu/\delta\wedge k} =0$ for $z\in [0,1]$.

Now we shall bound the residual terms in \eqref{eq:2.9}. By Lemma \ref{lem:A2},
$$\sup_{x\in (0,1)}x^2|\psi''(x)|<\infty$$ and hence
\begin{equation}
\label{R1}
\EE_z\Big|R^{(1)}_{\nu/\delta\wedge  k-1}\Big|\le
\EE_z\sum_{j=1}^{\nu/\delta \wedge k}
 \frac 1 2 \big|\psi''(X^\delta_{t_{j-1}})\big|\mu^2 \big(X^\delta_{t_{j-1}}\big)^2\delta^2\le C_1 \delta
 \EE_z (\nu \wedge k\delta ).
\end{equation}
Similarly,  $\sup_{x\in (0,1)} x^{p+1} |\psi''(x)|<\infty$ and  by Corollary \ref{CorB} (applied with $r=0$)
\begin{multline}\label{R2}
\EE_z \Big|R^{(2)}_{\nu/\delta\wedge  k-1}\Big|\le
\EE_z \sum_{j=1}^{\nu/\delta \wedge k}
\big|\psi''(X^\delta_{t_{j-1}})\big|\mu \sigma  \big(X^\delta_{t_{j-1}}\big)^{p+1} |\xi_j| \delta^{3/2}
\le\\
C_2 \EE_z
\sum_{j=1}^{\nu/\delta \wedge k}\big|
 \xi_j
\big|\delta^{3/2} \le
C_2 \sqrt{\delta}\Big(\EE_z
(\nu \wedge k\delta) +\delta\Big).
\end{multline}
To bound $\big|R^{(3)}_{\nu/\delta\wedge  k-1}\big|$, note that
by Lemma \ref{lem:A2},
$$
\big|\psi'''\big(\tilde X^\delta_{t_{j-1}}\big)\big|\le C_3\big|\tilde X^\delta_{t_{j-1}}\big|^{-2p-1}
$$
and thus
\begin{multline}\label{eq:2.12}
\Big|R^{(3)}_{\nu/\delta\wedge  k-1}\Big|\le
\sum_{j<\nu/\delta \wedge k}   \big|\psi'''\big(\tilde X^\delta_{t_{j-1}}\big)\big|
\Big|\mu X^\delta_{t_{j-1}}\delta + \sigma \big(X^\delta_{t_{j-1}}\big)^p\xi_j\sqrt{\delta}\Big|^3
\le \\
4\sum_{j<\nu/\delta \wedge k}    \big|\psi'''\big(\tilde X^\delta_{t_{j-1}}\big)\big|\bigg(
\Big|\mu X^\delta_{t_{j-1}}\delta\Big|^3 + \Big|\sigma \big(X^\delta_{t_{j-1}}\big)^p\xi_j\sqrt{\delta}\Big|^3
\bigg)\le \\
C_4 \sum_{j<\nu/\delta \wedge k}    \big(\tilde X^\delta_{t_{j-1}}\big)^{-2p-1}
\big( X^\delta_{t_{j-1}}\big)^{3p}(|\xi_j|^3+1)\delta^{3/2},
\end{multline}
where the latter inequality holds, since $ X^\delta_{t_{j-1}}\le 1$ on the set $\{j<\nu/\delta \wedge k \}$.
Since $\tilde X^\delta_{t_{j-1}}$ is between $X^\delta_{t_{j-1}}$ and $X^\delta_{t_j}$
\begin{equation}
\label{eq:2.13}
\big(\tilde X^\delta_{t_{j-1}}\big)^{-2p-1}\le
\big( X^\delta_{t_{j-1}}\big)^{-2p-1}
\vee
\big( X^\delta_{t_{j}}\big)^{-2p-1}.
\end{equation}
For $j< \nu/\delta$, we have $X^\delta_{t_{j-1}}\ge \delta^\beta$  and
thus on the set $\big\{X^\delta_{t_{j-1}}\le X^\delta_{t_j}\big\}$,
\begin{multline}
\label{eq:2.14}
\big(\tilde X^\delta_{t_{j-1}}\big)^{-2p-1}
\big( X^\delta_{t_{j-1}}\big)^{3p}\le
\big(X^\delta_{t_{j-1}}\big)^{-2p-1}
\big( X^\delta_{t_{j-1}}\big)^{3p}
=  \\
\big( X^\delta_{t_{j-1}}\big)^{-(1-p)}\le
\delta^{-\beta(1-p)}.
\end{multline}
On the set $\big\{X^\delta_{t_{j-1}}\ge  X^\delta_{t_j}\big\}$ we have
\begin{multline}\label{this}
\big(\tilde  X^\delta_{t_{j-1}}  \big)^{-2p-1}
 \big(X^\delta_{t_{j-1}}\big)^{3p}\le
 \big( X^\delta_{t_{j}}\big)^{-2p-1}
\big(X^\delta_{t_{j-1}}\big)^{3p}
= \\
\frac
{
\big( X^\delta_{t_{j-1}}\big)^{3p}
}
{\Big( X^\delta_{t_{j-1}}(1+ \mu \delta) + \sigma \big(X^\delta_{t_{j-1}}\big)^p\xi_j\sqrt{\delta}\Big)^{2p+1}}.
\end{multline}
Note that for $x>0$, $a>0$ and $b$, such that  $xa+x^pb>0$,
\begin{multline*}
\frac{x^{3p}}{(xa+x^pb)^{2p+1}}=
\frac{\big(x^{1-p}\big)^{2p}}{(x^{1-p}a+b)^{2p+1}}=
\frac{1}{a^{2p}}\frac{\big(x^{1-p}a+b-b\big)^{2p}}{(x^{1-p}a+b)^{2p+1}}\le \\
\frac{2^{2p-1}}{a^{2p}}\frac{(x^{1-p}a+b)^{2p}+|b|^{2p}}{(x^{1-p}a+b)^{2p+1}}=
\frac{2^{2p-1}}{a^{2p}}\left(
\frac{1}{x^{1-p}a+b}
+
\frac{|b|^{2p}}{(x^{1-p}a+b)^{2p+1}}
\right).
\end{multline*}
Applying this inequality to \eqref{this}  on the set $\Big\{|\xi_j|\le \delta^{\frac 1 2 \big(\beta(1-p)-\frac 1 2\big)}\Big\}$ we get
\begin{align*}
&\big(\tilde  X^\delta_{t_{j-1}}  \big)^{-2p-1}
 \big(X^\delta_{t_{j-1}}\big)^{3p} \le \\
&\frac{C_5}{\big(X^\delta_{t_{j-1}}\big)^{1-p}(1+\mu\delta)+\sigma \xi_j \delta^{1/2}}
+
\frac{C_5|\sigma \xi_j \delta^{1/2}|^{2p}}{\Big(\big(X^\delta_{t_{j-1}}\big)^{1-p}(1+\mu\delta)+\sigma \xi_j \delta^{1/2}\Big)^{2p+1}}\le \\
&
\frac{C_5}{\delta^{\beta(1-p)}(1+\mu\delta)-\sigma \delta^{\frac 1 2 \big(\beta(1-p)+\frac 1 2\big)} }
+
\frac{C_5 \sigma^2 \delta^{p \big(\beta(1-p)+\frac 1 2\big)} }{\Big(\delta^{\beta(1-p)}(1+\mu\delta)-\sigma \delta^{\frac 1 2 \big(\beta(1-p)+\frac 1 2\big)} \Big)^{2p+1}}\le \\
&
C_6 \delta^{-\beta(1-p)} \left(1
+
\delta^{p(1/2-\beta(1-p))}
\right)\le C_7 \delta^{-\beta(1-p)},
\end{align*}
where the  inequalities hold for all sufficiently small $\delta>0$ and
we used the bounds  $X^\delta_{t_{j-1}}\ge \delta^\beta$ and $\beta(1-p)<1/2$.

Consequently, on the set  $\big\{X^\delta_{t_{j-1}}\ge  X^\delta_{t_j}\big\}$
\begin{equation}\label{eq:2.15}
\big(\tilde  X^\delta_{t_{j-1}}  \big)^{-2p-1}
 \big(X^\delta_{t_{j-1}}\big)^{3p}
\le C_7 \delta^{-\beta(1-p)}+
\delta^{-\beta(2p+1)}\One{|\xi_1|\ge \delta^{\frac 1 2 \big(\beta(1-p)-\frac 1 2\big)}}
\end{equation}
Plugging the bounds \eqref{eq:2.15} and \eqref{eq:2.14} into  \eqref{eq:2.12} and
applying the Corollary \ref{CorB}, we obtain the
estimate
\begin{align*}
& \EE_z \Big|R^{(3)}_{\nu/\delta\wedge  k-1}\Big|\le \\
& \EE_z
C_8 \sum_{j<\nu/\delta \wedge k}
\bigg(
 \delta^{-\beta(1-p)}+
\delta^{-\beta(2p+1)}\One{|\xi_1|\ge \delta^{\frac 1 2 \big(\beta(1-p)-\frac 1 2\big)}}
\bigg)(|\xi_j|^3+1)\delta^{3/2}\le \\
&
2 C_8 \EE_z(\nu \wedge k\delta)
 \EE_z
\bigg(
 \delta^{\gamma }+
\delta^{-\frac 3 2 \frac 1 {1-p}}\One{|\xi_1|\ge \delta^{-\frac 1 2 \gamma}}
\bigg)(|\xi_j|^3+1) +C_9 \delta^{3/2},
\end{align*}
where $\gamma = 1/2-\beta(1-p)>0$.
Using the Gaussian tail estimate
$$
\E |\xi_1|^{2m} \one{|\xi_1|\ge a}\le  \sqrt{(4m-1)!!}  \frac 1 {\sqrt{a}} e^{-\frac 1 4 a^2},\quad m\ge 1, \ a>0,
$$
we get
$$
\EE_z \Big|R^{(3)}_{\nu/\delta\wedge  k-1}\Big|\le C_{10} \EE_z(\nu \wedge k\delta ) \delta^{\gamma} + C_{10}\delta^{3/2},
$$
 which
along with \eqref{R1} and \eqref{R2} yields the bound
\begin{equation}\label{Rbnd}
\EE_z \big|R_{\nu/\delta\wedge  k-1}\big| \le C_{11} \EE_z (\nu\wedge k\delta) \delta^{\gamma}+ C_{11}\delta^{3/2}.
\end{equation}
Finally, by Corollary \ref{CorB},
$$
\EE_z\big|\xi_{\nu/\delta\wedge k}\big|\vee \EE_z\big(\xi_{\nu/\delta\wedge k}\big)^2 \le
\EE_{z}\big(\nu/\delta\wedge k\big)^{1/4} + C_{12} \le
\Big(\EE_{z}\big(\nu/\delta\wedge k\big)\Big)^{1/4} + C_{12}
$$
and hence by Lemma \ref{lem:A2},
$$
\EE_z\Big|M_{\nu/\delta \wedge k}-M_{\nu/\delta \wedge k-1}\Big|\le
C_{13}\delta^{1/4} \Big(\EE_z\big(\nu \wedge k \delta\big)\Big)^{1/4}\le
C_{13}\delta^{1/4} \Big(1+\EE_z\big(\nu \wedge k \delta\big)\Big),
$$
and consequently
\begin{equation}
\label{rbnd}
\EE_z r(\delta) \le \psi\big(X^\delta_{\nu\wedge k\delta }\big) +
C_{13}\delta^{1/4} \Big(1+\EE_z\big(\nu \wedge k \delta\big)\Big).
\end{equation}
Plugging the estimates \eqref{Sbnd}, \eqref{Rbnd} and \eqref{rbnd} into \eqref{eq:2.9}, we obtain the bound
$$
\EE_z\big(\nu \wedge k \delta\big)
\Big(1-C\delta^{\gamma}-C \delta^{1/4}\Big)\le C\delta^{1/4} + \psi(z) +\sup_{y\in [0,1]}\psi(y)
 <\infty.
$$
By the monotone convergence, the latter implies
$$
\EE_z \nu\le \frac{C\delta^{1/4} + \psi(z) +\sup_{y\in [0,\delta^\beta]}\psi(y)
 }{1-C\delta^{\gamma}-C \delta^{1/4}},
$$
and by continuity of $\psi$,
$$
\varlimsup_{\delta\to 0}\sup_{z\in [0,\eps]}\EE_z \nu \le \sup_{z\in [0,\eps]}\psi(z)\xrightarrow{\eps\to 0}0,
$$
verifying \eqref{une}.

\subsubsection*{Proof of \eqref{doux}}

Consider the function $\varphi(x):= \P_x\big(\tau_0> \tau_1\big)$, whose domain we extend to the whole real line by continuity,
setting $\varphi(x)=0$ for $x<0$ and $\varphi(x)=1$ for $x> 1$. The process $\varphi(X^\delta_j)$, $j\ge 0$ satisfies the
decomposition \eqref{eq:2.9}, with $\psi$ replaced by $\varphi$.
Taking into account that $(\mathcal{L}\varphi)(x)=0$
for $x\in [0,1]$ and the bounds from the Lemma \ref{lem:A1}, a calculation, similar to the one in the preceding subsection, shows that
$$
\EE_z \varphi(X^\delta_{\nu}) - \varphi(z)\le  \EE_z\Big(\varphi\big(X^\delta_{\nu }\big) -\varphi\big(X^\delta_{\nu -\delta}\big)\Big) + C\delta^{1/4}+
\Big(C\delta^{\gamma}+
C\delta^{1/4}
\Big)\EE_z \nu.
$$
On the other hand,
\begin{align*}
&\EE_z\varphi(X^\delta_{\nu})=  \EE_z\big(\nu_{\delta^\beta}>\nu_{1}\big)\varphi\big(X_{\nu_{1}}\big)
+
\EE_z\big(\nu_{\delta^\beta}<\nu_{1}\big)\varphi\big(X_{\nu_{\delta^\beta}}\big) =\\
&\EE_z\big(\nu_{\delta^\beta}>\nu_{1}\big)\left(
\varphi(1)
-\varphi\big(X_{\nu_{\delta^\beta}}\big)
\right) +\EE_z \varphi\big(X_{\nu_{\delta^\beta}}\big)\ge
\PP_z\big(\nu_{\delta^\beta}>\nu_{1}\big)\bigg(
1
-\sup_{y\in [0,\delta^\beta]}\varphi(y)
\bigg)
\end{align*}
and thus, for sufficiently small $\delta>0$,
\begin{multline*}
\sup_{z\in [0,\eps]}\PP_z\big(\nu_{\delta^\beta}>\nu_{1}\big)\le
C\left(
1
-\sup_{y\le \delta^\beta}\varphi(y)
\right)^{-1}\times \\ \left(\sup_{z\in [0,\eps]}\varphi(z)+
 \sup_{z\in [0,\eps]}\EE_z\Big(\varphi\big(X^\delta_{\nu }\big) -\varphi\big(X^\delta_{\nu -\delta}\big)\Big)
+
\delta^{\gamma\wedge 1/4}\big(1+\sup_{z\in [0,\eps]}\EE_z \nu\big) \right).
\end{multline*}

Since $\varphi$ is continuous on $[0,1]$ and $\varphi(0)=0$,
$$
\varlimsup_{\delta\to 0}\sup_{z\in [0,\eps]}\PP_z\big(\nu_{\delta^\beta}>\nu_{1}\big) \le
C \inf_{z\in [0,\eps]}\varphi(z)\xrightarrow{\eps\to 0}0,
$$
which verifies  \eqref{doux}.

\end{proof}

\begin{rem}
The condition $\beta <\frac {1/2}{1-p}$ originates in the estimate \eqref{eq:2.14}, which is plugged into \eqref{eq:2.12}.
The principle difficulty is that the term
$
\big(\tilde X^\delta_{t_{j-1}}\big)^{-2p-1}
\big( X^\delta_{t_{j-1}}\big)^{3p}
$
cannot be effectively controlled for greater values of $\beta$ as $\delta\to 0$.     For example,
it is not clear how to bound the right hand side of \eqref{eq:2.12} already for  $\beta=1$ and $p=1/2$.
\end{rem}
\section{Numerical experiments}\label{sec-num}

In insurance, one is often interested in calculating the probability of ruin by a particular time $t>0$.
For the diffusion  \eqref{cev},  this probability can be found explicitly and  for $p=1/2$, it has a particularly simple form:
\begin{equation}
\label{Pfla}
\P_x\big(\tau_0(X)\le t\big)=\begin{cases}
\exp\left(-\dfrac{2x}{\sigma^2}\dfrac 1 {t}\right), & \mu=0 \\
\exp\left(-\dfrac{2x}{\sigma^2}\dfrac{\mu e^{\mu t}}{e^{\mu t}-1}\right), & \mu \ne 0
\end{cases}
\end{equation}
Note that $\tau_0(X)$ has an atom at $\{+\infty\}$ when $\mu>0$:
$$
\P_x(\tau_0(X)<\infty)=e^{-2x\mu/\sigma^2}<1.
$$

\begin{figure}
  \includegraphics[scale=0.6]{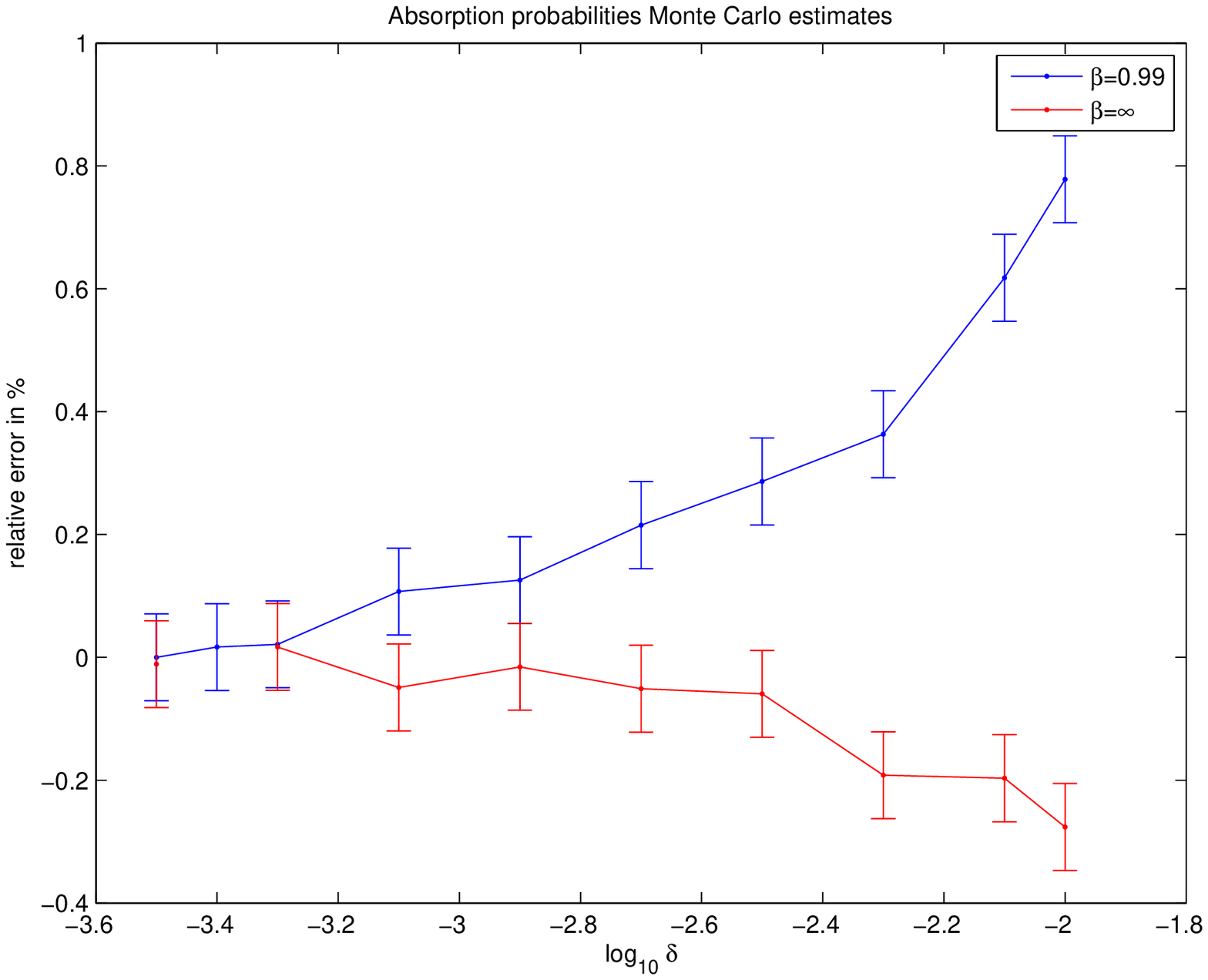 }\\
  \caption{\label{fig1} CEV model with $\mu=0$, $\sigma=1$, $p=1/2$, $x=1$ and $t=5$. The corresponding exact value of
  the absorption probability is $\P_x\big(\tau_0(X)\le t\big)=0.6703...$ }
\end{figure}

Figure \ref{fig1} depicts the results of the Monte Carlo simulation, in which the probability
of absorption  \eqref{Pfla} has been estimated for particular values of the
model parameters, using $M=10^7$ i.i.d trajectories, generated by the Euler-Maruyama
algorithm \eqref{em} and \eqref{lin}. The  relative estimation errors
$$
\mathrm{Err} : =\frac{\widehat{\P}_x\big(\tau_{\delta^\beta}(X^\delta)\le t\big)-\P_x\big(\tau_0(X)\le t\big)}{\P_x\big(\tau_0(X)\le t\big)}
\times 100\%
$$
are plotted versus $\delta$ (in the $\log$ scale), along with the $99\%$ confidence intervals, based
on the CLT approximation. The results appear to  be practically adequate: for example, the accuracy of $0.1\%$ is
obtained already with $\delta=10^{-3}$. The positive bias of the error is not surprising, since the earlier
absorption is more probable for the larger  threshold $\delta^\beta$.

The simulation results also indicate in favor of the convergence
$$
\lim_{\delta\to 0}\P_x\big(\tau_{0}(X^\delta)\le T\big)=
\P_x\big(\tau_{0}(X)\le T\big),
$$
which remains a plausible conjecture.

\appendix
\section{Some properties of the CEV diffusion}\label{A}

In this section we summarize some asymptotic estimates  of the absorption times
for the diffusion \eqref{cev} near the boundary $\{0\}$, which are used in the proof of Lemma \ref{lem3}.
Recall that
$$
\varphi(x):= \P_x\big(\tau_0> \tau_1\big), \quad \text{and} \quad \psi(x):=\E_x \tau_0\vee \tau_1
$$
are the solutions of the following problems respectively (see e.g. \cite{KT2}):
\begin{equation}
\label{prob}
(\mathcal{L}\varphi)(x) = 0, \quad x\in (0,1), \quad \varphi(0)=0, \ \varphi(1)=1
\end{equation}
and
\begin{equation}
\label{expect}
(\mathcal{L}\psi)(x)= -1, \quad x\in (0,1), \quad \psi(0)=\psi(1)=0,
\end{equation}
where $\mathcal{L}$ is the
operator, given by \eqref{genX}. The solutions  in the class of continuous functions on $[0,1]$, which
are twice differentiable on $(0,1)$ are given by the formulas \cite{KT2}:
\begin{equation}
\label{prob-fla}
\varphi(x)= \frac{S(x)-S(0)}{S(1)-S(0)}, \quad x\in [0,1],
\end{equation}
and
\begin{multline}
\label{expect-fla}
\psi(x) =
\varphi(x)\int_x^1 \Big(S(1)-S(y)\Big)m(y)dy
+\\
2\big(1-\varphi(x)\big)\int_0^x \Big(S(y)-S(0)\Big)m(y)dy,
\end{multline}
where
\begin{align*}
& s(y)= \exp\left(-\int^x_{0} \frac{2\mu }{\sigma^2 }z^{1-2p}dz\right) & \text{(scale density)} \\
& S(y)-S(x)=\int_x^y s(z)dz & \text{(scale measure)} \\
& m(y)=\dfrac{1}{\sigma^2 y^{2p} s(y)} & \text{(speed density)}.
\end{align*}

\begin{lem}\label{lem:A1}
The function $\varphi$ is smooth on $(0,1)$, and
$$
\sup_{x\in (0,1)}\Big(\big|\varphi'(x)\big|+x^{2p-1}\big|\varphi''(x)\big|+ x^{2p}\big|\varphi'''(x)\big|\Big)<\infty.
$$
\end{lem}

\begin{proof}
The scale density $s(x)=\exp\left(-\frac {2\mu}{\sigma^2(2-2p)}x^{2-2p}\right)$ is smooth on $(0,1)$, and
$$
\sup_{x\in (0,1)}\big|\varphi'(x)\big| = \frac{1}{S(1)-S(0)}\sup_{x\in [0,1]}s(x)<\infty.
$$
From \eqref{prob} we get
$
\varphi''(x) = -\dfrac{2\mu x \varphi'(x)}{\sigma^2 x^{2p}}
$
and thus
$$
\sup_{x\in (0,1)}x^{2p-1}\big|\varphi''(x) \big|\le \dfrac{2|\mu|  }{\sigma^2 }\sup_{x\in (0,1)}\big|\varphi'(x)\big|<\infty.
$$
Differentiating \eqref{prob}, we get
\begin{equation}\label{eq:A.5}
\mu  \varphi'(x) +\mu x \varphi''(x)+  \sigma^2 p x^{2p-1}\varphi''(x)+
\frac 1 2 \sigma^2 x^{2p}\varphi'''(x)
=0,
\end{equation}
and consequently
$$
\sup_{x\in (0,1)}x^{2p}\big|\varphi'''(x)\big|<\infty.
$$
\end{proof}

\begin{lem}\label{lem:A2}
The function
$\psi$ is smooth on $(0,1)$ and
$$
\sup_{x\in (0,1)}\Big(x^{2p-1}\big|\psi'(x)\big|+x^{2p}\big|\psi''(x)\big|+ x^{2p+1}\big|\psi'''(x)\big|\Big)
<\infty, \quad\text{for\ } p\in (1/2,1)
$$
and
$$
\sup_{x\in (0,1)}\Big(\frac 1 {\log (1/x) \vee 1}\big|\psi'(x)\big|+x^{2p}\big|\psi''(x)\big|+ x^{2p+1}\big|\psi'''(x)\big|\Big)
<\infty,\ \text{for\ } p =1/2.
$$
\end{lem}
\begin{proof}
We have
\begin{multline*}
\frac 1 2 \psi'(x)=
\varphi'(x)\int_x^1 \Big(S(1)-S(y)\Big)m(y)dy-\varphi(x) \Big(S(1)-S(x)\Big)m(x)
\\
-\varphi'(x)\int_0^x \Big(S(y)-S(0)\Big)m(y)dy+\big(1-\varphi(x)\big)\Big(S(x)-S(0)\Big)m(x),
\end{multline*}
and, since $S(x)$ is increasing and $s(x)\le 1$,
\begin{multline*}
\frac 1 2 \big|\psi'(x)| \le
|\varphi'(x)|\int_x^1 m(y)dy+ \frac{1}{S(1)-S(0)}x m(x)
\\
+\big|\varphi'(x)\big| \int_0^x y m(y)dy+xm(x).
\end{multline*}
For $p=1/2$, it follows that
$$
\sup_{x\in (0,1)}\frac{1}{\log(1/x)\vee 1}\big|\psi'(x)|<\infty,
$$
and for $p\in (1/2,1)$
$$
\sup_{x\in (0,1)}x^{2p-1}\big|\psi'(x)|<\infty.
$$
Now \eqref{expect} implies
$$
\psi''(x) = -\frac{2\mu }{\sigma^2}x^{1-2p}\psi'(x)-\frac{2}{\sigma^2}x^{-2p}
$$
and hence for $p\in [1/2,1)$,
$$
\sup_{x\in (0,1)}x^{2p}|\psi''(x)|<\infty.
$$
Further, differentiating \eqref{expect} we see that $\psi$ satisfies \eqref{eq:A.5} as well and hence
$$
\sup_{x\in (0,1)}x^{2p+1}|\psi''(x)|<\infty,
$$
which verifies the claim.
\end{proof}

\section{An inequality}

\begin{lem}\label{lemB}
Let $(\eta_j)_{j\in \mathbb{N}}$ be a sequence of random variables and $N$ be an integer valued
random variable. Then for constants $\alpha >0$ and $c>0$ and an integer $k$
$$
\E |\eta_{N\wedge k}| \le c\E\big(N\wedge k\big)^\alpha +
\sum_{j=0}^{k} \E |\eta_j|\one{|\eta_j|> c j^\alpha}.
$$
\end{lem}
\begin{proof}
For  a fixed integer $k\ge 1$ and a real number $\alpha>0$
\begin{align*}
&\E |\eta_{N\wedge k}| = \E \sum_{j=0}^{k-1} |\eta_j|\one{N=j} + \E |\eta_{k}| \one{N\ge   k}\le\\
&
c\sum_{j=0}^{k-1} j^\alpha \P(N=j)
+ ck^\alpha \P(N\ge  k) +\E \sum_{j=0}^{k-1} |\eta_j|\one{|\eta_j|> cj^\alpha}
+\E |\eta_{k}|\one{|\eta_{k}|> ck^\alpha}\le \\
&c\E (N\wedge k)^\alpha + \sum_{j=0}^{k}\E |\eta_j|\one{|\eta_j|> c j^\alpha}
\end{align*}
\end{proof}

\begin{cor}\label{CorB}
Let $(\xi_j)_{j\in \mathbb{N}}$ be an i.i.d. $N(0,1)$ sequence and  $N$ be an integer valued
random variable. Then for any $\alpha>0$, $p>0$  and integer $k$
\begin{equation}
\label{b1}
\E |\xi_{N\wedge k}|^p\le \E\big(N\wedge k\big)^\alpha + C_{\alpha,p},
\end{equation}
with a constant $C_{\alpha, p}$ depending only on $\alpha$ and $p$. Moreover, for any $r\ge 0$, the sum
$S_n = \sum_{j=0}^n |\xi_j|^p\one{|\xi_j|\ge r}$
satisfies
\begin{equation}
\label{b2}
\E S_{N\wedge k}\le  2 \E |\xi_1|^p\one{|\xi_1|\ge r} \E (N\wedge k)+ C_p,
\end{equation}
with a constant $C_p$, depending only on $p$.
\end{cor}

\begin{proof}
Applying Lemma \ref{lemB} with $c:=1$ and $\eta_j:=|\xi_j|^p$, we obtain the inequality \eqref{b1} with
\begin{multline*}
C_{\alpha,p}:=\sum_{j=1}^\infty \E |\xi_j|^p\one{|\xi_j|^p >  j^\alpha}=\sum_{j=1}^\infty
\sqrt{\E |\xi_1|^{2p}}\sqrt{\P(\xi_1 >  j^{\alpha/p})} \le\\
 \sqrt{(2p-1)!!} \sqrt{\frac{2}{\pi}}
\sum_{j=1}^\infty \frac{1}{j^{\alpha/2p}}e^{-\frac 1 4j^{2\alpha/p}}<\infty.
\end{multline*}

Further,  define  $\zeta_j:=|\xi_j|^p\one{|\xi_1|\ge r}-\E|\xi_1|^p\one{|\xi_1|\ge r}$, so that $\E\zeta_1=0$ and $\zeta_i$'s are i.i.d. and
$\E |\zeta_1|^{2m}<\infty$ for any $m\ge 1$. It follows that
\begin{equation}
\label{simb}
\E\left(\frac 1 j \sum_{i=1}^j \zeta_i\right)^{2m} \le C_m/ {j^m},
\end{equation}
with a constant $C_m$.

Applying  Lemma \ref{lemB}  with $\alpha:=1$, $c:=2\E|\xi_1|^p\one{|\xi_1|\ge r}$ and $\eta_j:=S_{j}$, we get
\begin{align*}
\E S_{N\wedge k}\le  & 2 \E |\xi_1|^p\one{|\xi_1|\ge r} \E (N\wedge k)+
\sum_{j=0}^{k} \E S_j\one{S_j> 2j\E|\xi_1|^p} \le \\
&2 \E |\xi_1|^p\one{|\xi_1|\ge r} \E (N\wedge k)+
\sum_{j=0}^{k} \sqrt{\E S^2_j}\sqrt{\P\left(\frac 1 j \sum_{i=1}^j \zeta_i > c/2\right)}\le \\
&2 \E |\xi_1|\one{|\xi_1|\ge r} \E (N\wedge k)+
C \sum_{j=0}^{k} j\frac 1{c^5}\sqrt{\E\left(\frac 1 j \sum_{i=1}^j \zeta_i \right)^{10}}\le \\
&
2 \E |\xi_1|\one{|\xi_1|\ge r} \E (N\wedge k)+
\frac{C \sqrt{C_{10}}}{c^5} \sum_{j=0}^{\infty} j^{-3/2}.
\end{align*}
\end{proof}


\def\cprime{$'$} \def\cprime{$'$} \def\cprime{$'$}

\end{document}